\documentclass{amsart}
\usepackage{amssymb}
\usepackage[all]{xy}

\newcommand{\Z}{\mathbf{Z}}
\newcommand{\Ps}{\mathbf{P}}
\newcommand{\As}{\mathbf{A}}

\newcommand{\C}{\mathbf{C}}

\newcommand{\cC}{\mathcal{C}}

\newtheorem{theorem}{Theorem}[section]
\newtheorem{lemma}[theorem]{Lemma}
\newtheorem{proposition}[theorem]{Proposition}

\theoremstyle{definition}
\newtheorem{definition}[theorem]{Definition}
\newtheorem{notation}[theorem]{Notation}

\theoremstyle{remark} 
\newtheorem{remark}[theorem]{Remark}

\DeclareMathOperator{\prim}{prim}

\DeclareMathOperator{\Jac}{Jac}

\DeclareMathOperator{\lcm}{lcm}

\title{Curves with rational families of quasi-toric relations}
\author{Remke Kloosterman}
\address{Universit\`a degli Studi di Padova,
Dipartimento di Matematica,
Via Trieste 63,
35121 Padova, Italy}
\date{\today}
\begin{document}

\begin{abstract} 
We investigate which plane curves admit rational families of quasi-toric relations. This extends previous results of Takahashi and Tokunaga in the positive case and of the author in the negative case.
\end{abstract}
\maketitle
\section{introduction}\label{secInt}

Let $F_0,F_1,F_2\in \C[x_0,x_1,x_2]$ be  homogeneous polynomials, without a common factor. Let $p_0,p_1,p_2$ be three positive integers. A \emph{quasi-toric relation} of type $(p_0,p_1,p_2)$ of $(F_0,F_1,F_2)$ is a triple $(h_0,h_1,h_2)$ of nonzero homogeneous polynomials such that
\[  F_0 h_0^{p_0}+F_1h_1^{p_1}+F_2h_2^{p_2} =0.\]
There is an obvious equivalence relation on these structures, see Definition~\ref{defEqu}.

Examples of triples $(F_0,F_1,F_2)$ with infinitely many equivalence classes of quasi-toric relations can be found, e.g., if $(p_0,p_1,p_2)=(2,3,6)$ then the quasi-toric relations form an abelian group. If we pick $F_0=F_1=1; F_2=f^2+g^3$ for general forms $f,g$ of degree $3k, 2k$  then the group of $(2,3,6)$-quasi-toric relations contains $\Z^ 2$ as a subgroup.
However, these relations cannot be parametrised by a rational curve.

In \cite{TT} Takahashi ad Hiro-o Tokunaga constructed examples of triples $(F_0,F_1,F_2)$  for which there is a family of quasi-toric relations parametrised by a rational curve. In this case the type is  $(2,2,2n+1)$ with $n\in \{1,2,3\}$. On the other hand we argued in \cite{TorDec} that for the types of the shape $(n,m,\lcm(n,m))$, with $n\geq 2,m> 2$  no such families can exist.
The main aim of this paper is to describe when positive dimensional  families of quasi-toric relations exists and when not. 

Firstly we investigate for which choices of types $(p_0,p_1,p_2)$ and $(q_0,q_1,q_2)$ there exists  a one-to-one correspondence between quasi-toric relations of type $(p_0,p_1,p_2)$ and of type $(q_0,q_1,q_2)$. More precisely, given a triple $(p_0,p_1,p_2)$ of positive integers, set $g=\gcd(p_0,p_1,p_2)$. For $\{i,j,k\}=\{0,1,2\}$ define $q_i=g\gcd(p_i/g,p_jp_k/g^2)$.
  We call the type $(q_1,q_2,q_3)$ reduced. We show that there is a natural one-to-one correspondence between equivalence classes of quasi-toric relations of type $(p_0,p_1,p_2)$ and of type $(q_0,q_1,q_2)$, see Proposition~\ref{prpEqu}. Hence one may restrict to the reduced types. 
Quasi-toric relations of reduced type can be interpreted as certain special $\C(x_1,x_2)$-rational points on a curve $\mathcal{C}/\C(x_1,x_2)$, see Proposition~\ref{prpCur}. We will show in Proposition~\ref{prpInf} that if there is a positive dimensional family of quasi-toric relations then the curve $
\cC/\C(x_1,x_2)$ has genus 0. This happens precisely whenever  one of the $q_i$ equals 1, or all the $q_i$ equal 2. Hence our main result is
\begin{theorem} Let $F_0,F_1,F_2$ be three homogeneous polynomials.
Let $(p_0,p_1,p_2)$ be a reduced type, such that $p_0\leq p_1\leq p_2$. Then there exists a positive dimensional family of quasi-toric relations of $(p_0,p_1,p_2)$, parametrized by a rational variety if and only if one of the following occurs
\begin{enumerate}
\item $p_0=p_1=p_2=2$ and there is at least one quasi-toric relation or
\item $p_0=1$ and $p_1$ divides $\deg(F_1)-\deg(F_2)$.
\end{enumerate}
\end{theorem}

The relations found in \cite{TT} are of type $(2,2,2k+1)$ and the associated reduced type is $(2,2,1)$, hence this is consistent with our results. However, 
one should note that the relations  in \cite{TT} are of a special shape: If one starts with a quasi-toric relation of type $(2,2,1)$ and one applies the correspondence to obtain a relation of type $(2,2,2n+1)$ then for most equivalences classes one has that  every representative $(h_0,h_1,h_2)$ satisfies $\gcd(h_0,h_1,h_2)\neq 1$. 
The examples of Takahashi and Tokunaga do satisfy $\gcd(h_0,h_1,h_2)=1$. It is this latter fact which makes their construction very interesting.

One should remark that quasi-toric relations are useful to construct Zariski pairs. For this one uses the Alexander polynomial of the curve $F_0F_1F_2=0$. 
Most proofs of this connection between quasi-toric relations and the Alexander polynomial work under a technical assumption, which tend to imply that the curve $\cC/\C(x_1,x_2)$ has positive genus. In particular, it does not work in the case of families of quasi-toric relations constructed by Takahashi and Tokunaga.

\section{Quasi-toric relations}\label{secQua}
\begin{definition}\label{defQua}
Let $F_0,F_1,F_2\in \C[x_0,x_1,x_2]$ be  homogeneous polynomials. Let $p_0,p_1,p_2$ be three positive integers. A \emph{quasi-toric relation} of type $(p_0,p_1,p_2)$ of $(F_0,F_1,F_2)$ is a triple $(h_0,h_1,h_2)$ of nonzero homogeneous polynomials such that
\begin{equation}\label{eqnTorRel} F_0 h_0^{p_0}+F_1h_1^{p_1}+F_2h_2^{p_2} =0.\end{equation}
\end{definition}

\begin{definition}\label{defEqu}
Let $F_0,F_1,F_2\in \C[x_0,x_1,x_2]$ be  homogeneous polynomials. Let $p_0,p_1,p_2$ be three positive integers. Let $d=\lcm(p_0,p_1,p_2)$. For $i\in \{0,1,2\}$ let $w_i=d/p_i$.

Two {quasi-toric relations} $(g_0,g_1,g_2),(h_0,h_1,h_2)$ of type $(p_0,p_1,p_2)$ are \emph{equivalent} if there exists nonzero homogenous forms $u,v$ such that
\[ u^{w_i} g_i=v^{w_i} h_i\]
for $i=0,1,2$.
\end{definition}

 In \cite{CogLib} the authors defined an equivalence relation on the sextuple \[(h_0,h_1,h_2,F_0,F_1,F_2).\] For our purposes it turns out to be preferable to fix $F_0,F_1,F_2$.

\begin{lemma} \label{lemEquA}
Let $F_0,F_1,F_2\in \C[x_0,x_1,x_2]$ be  homogeneous polynomials. Let $p_0,p_1,p_2$ be three positive integers. Let $m$ be an integer such that $m\mid p_2$. 
Then the map $(h_0,h_1,h_2) \to (h_0,h_1,h_2^m)$ defines a map from equivalence class of quasi-toric relations of type $(p_0,p_1,p_2)$ to equivalence classes of quasi-toric relations of type $(p_0,p_1,p_2/m)$
\end{lemma}

\begin{proof} Immediate.
%There is an obvious map from the $(p_0,p_1,p_2)$ quasi-toric relations to the $(p_0,p_1,q_2)$ quasi-toric relations:
%Suppose $(h_0,h_1,h_2)$ is a $(p_0,p_1,p_2)$ quasi-toric relation then $(h_0,h_1,h_2^{m})$ is a $(p_0,p_1,q_2)$ quasi-toric relation and it is obvious that this maps equivalence classes to equivalence classes.

%Moreover if $(w_0,w_1,w_2)$ are the weights for $(p_0,p_1,p_2)$ then $(w_0/m,w_1/m,w_2)$ are the weights for $(p_0,p_1,q_2)$ and hence equivalence classes are mapped to equivalence classes.
%
%
%
%Let $d=uv_{01}v_{02}v_{12}k_0k_1k_2$.
%Set $w_0=d/p_0=v_{12}k_1k_2$ $w_1=d/p_1=v_{02}k_0k_2$, $w_2=d/p_2=v_{01}k_0k_1$.
%
%Another element in the equivalence class of $(h_0,h_1,h_2)$ is of the form $(f^{w_0}h_0,f^{w_1} h_1,f^{w_2}h_2)$, with $f$ a quotient of two homogeneous forms. This relation is mapped to $(f^{w_0}h_0,f^{w_1} h_1,f^{mw_2}h_2^m)$. However, for these structures of type $(p_0,p_1,p_2)$
%$d=uv_{01}v_{02}v_{12}k_0k_1 k_2/m$ and weights $(w_0/m, w_1/m, w_2)$, hence the above relation is equivalent with $(h_0,h_1,h_2^m)$.
\end{proof}

\begin{notation}\label{notExp} For three positive integers $(p_0,p_1,p_2)$ set 
\[ r=\gcd(p_0,p_1,p_2), d=\lcm(p_0,p_1,p_2),  s_{ij}=\gcd(p_i/r,p_j/r),\mbox{ and }t_i=p_i/(rs_{ij}s_{ik}),\] where $\{i,j,k\}=\{0,1,2\}$.
Then 
 \[p_0=rs_{01}s_{02}t_0, p_1=rs_{01}s_{12}t_1 p_2=rs_{02}s_{12}t_2 \mbox{ and } d=rs_{01}s_{02}s_{12}t_0t_1t_2. \]
Moreover, let
\[w_i:=\frac{d}{p_i}=s_{jk}t_jt_k\]
with $\{i,j,k\}=\{0,1,2\}$.
\end{notation}

\begin{lemma}\label{lemEqB}
Let $F_0,F_1,F_2\in \C[x_0,x_1,x_2]$ be  homogeneous polynomials. Let $p_0,p_1,p_2$ be three positive integers. Let $m$ be an integer such that $m\nmid p_0p_1$ and $m\mid p_2$. Then $\gcd(s_{01}t_0t_1,m)=1$.
Let $k$ be an integer such that $k(s_{01}t_0t_1) \equiv -1 \bmod m$.

Then the map \[(h_0,h_1,h_2) \to (h_2^{ks_{12}t_1 t_2/m } h_0,h_2^{ks_{02}t_0t_2/m} h_1,h_2^{(1+ks_{01}t_0t_1)/m})\]
 defines a map from equivalence class of quasi-toric relations of type $(p_0,p_1,p_2/m)$ to equivalence class of quasi-toric relations of type $(p_0,p_1,p_2)$
\end{lemma}

\begin{proof}
Take a quasi-toric relation $(h_0,h_1,h_2)$ of type $(p_0,p_1,p_2/m)$ then
\[ h_0^{p_0}F_0+h_1^{p_1}F_1+h_2^{p_2/m} F_2=0.\]
Multiply this equation with $h_2^{kd/m}$ then we find the following quasi-toric relation of type $(p_0,p_1,p_2)$
\[(h_2^{kd/(mp_0)} h_0, h_2^{kd/(mp_1)}h_1, h_2^{(1+(kd)/p_2)/m}) ,\]
provided that  the exponent of $h_2$ in each of the entries is an integer.
Recall that $kd/(mp_0)=k s_{12}t_1t_2/m$, $kd/(mp_1)=k s_{02}t_0t_2/m$ and $kd/p_2=ks_{01}t_0t_1$.  Since $m$ divides $t_2$ and  we choose $k$ such that $1+ks_{01}t_0t_1$ is divisible by $m$ we obtain that the exponents of $h_2$ are indeed integers. %Hence our map sends quasi-toric relation of type $(p_0,p_1p_2/m)$ to quasi-toric relations of type $(p_0,p_1,p_2)$.

The integer $k$ is unique modulo $m$. If $k_1=k+vm$ then the exponents of $h_2$ differ by $(v(s_{12}t_1t_2),v(s_{02}t_0t_2),v(s_{01}t_0t_1))=v(w_0,w_1,w_2)$, hence the obtained quasi-toric relation is in the same equivalence class.
Similarly, one easily checks that an equivalent quasi-toric relation of $(h_0,h_1,h_2)$ is mapped to an element of the same equivalence class.
\end{proof}
%\begin{proof}
%Let $u=\gcd(p_0,p_1,p_2)$, $v_{ij}=\gcd(p_i/u,p_j/u)$ and take $k_i$ such that  \[p_0=uv_{01}v_{02}k_0, p_1=uv_{01}v_{12}k_1 p_2=uv_{02}v_{12}k_2.\]
%Then $m$ is a divisor of $k_2$. Let $m_2=k_2/m$.%From a transitivity argument it  suffices to prove the result for $m=k_2$. 
%Set $q_2=p_2/m$.%=uv_{02}v_{12}$.
%Then $d=uv_{01}v_{02}v_{12}k_0k_1k_2$.
%Set $w_0=v_{12}k_1k_2$ $w_1=v_{02}k_0k_2$, $w_2=v_{01}k_0k_1$.

%Suppose now $g_0,g_1,g_2$ is $(p_0,p_1,q_2)$ quasi-toric relation of $F_0,F_1,F_2$
%For any positive integer $t$, let $r=uv_{01}v_{02}v_{12}k_0k_1m_2 t$.

%Then also $g_0g_2^{r/p_0}, g_0g_2^{r/p_1}, g_2^{1+r/q_2}$ is a $p_0,p_1,q_2$ quasi-toric relation. Recall that $r/q_2=v_{01}k_0k_1t\in \Z$.

%We want to determine the values of $t$ for which this gives rise to a relation of type $(p_0,p_1,p_2)$ i.e., when $k_2$ divides $1+r/q_2$. This happens precisely when $t(v_{01}k_0k_1)\equiv -1 \bmod k_2$. By construction $\gcd(v_{01},k_2)=\gcd(k_0,k_2)=\gcd(k_1,k_2)=1$. Hence there exist infinitely many integers $t$ satisfying the above congruences. 
%
%\end{proof}

%Then there is a one-to-one correspondence between equivalence classes of quasi-toric relations of type $(p_0,p_1,p_2)$ and of type $(p_0,p_1,p_2/m)$.

\begin{proposition}\label{prpEqu}
Let $F_0,F_1,F_2\in \C[x_0,x_1,x_2]$ be  homogeneous polynomials. Let $p_0,p_1,p_2$ be three positive integers. Let $m$ be an integer such that $m\nmid p_0,p_1$ and $m\mid p_2$. Then there is a one-to-one correspondence between equivalence classes of quasi-toric relations of type $(p_0,p_1,p_2)$ and of type $(p_0,p_1,p_2/m)$.
\end{proposition}

\begin{proof} 
It suffices to prove the maps constructed in the previous two lemmata are inverse maps to each other.

If we start with a relation  $(h_0,h_1,h_2)$ of type $(p_0,p_1,p_2)$  and we apply both maps then we obtain that the $(p_0,p_1,p_2)$ quasi-toric relation $(h_0,h_1,h_2)$ is mapped to
\[\left(h_0h_2^{rs_{12}t_1t_2}, h_1h_2^{rs_{02}t_0t_2},h_2^{k_2(1+ks_{01}t_0t_1)/k_2}\right).\]
The exponents of $ h_2$ at each of the components are $tw_0, tw_1, 1+tw_2$, hence we may apply the equivalence relation with $u=1,v=h_2^t$ and obtain that the above relation is equivalent with\[ (h_0,h_1,h_2).\]

Similarly for a relation  $(g_0,g_1,g_2)$ of type $(p_0,p_1,p_2/m)$  we obtain that this relation is mapped to
\[ (g_0g_2^{krs_{12}t_1t_2/m}, g_1g_2^{krs_{02}t_0t_2/m},g_2^{(1+krs_{01}t_0t_1)}).\]
For this type the weights for the equivalence relation are \[(s_{12}t_1t_2/m,s_{02}t_0t_2/m, s_{01}t_0t_1).\] By taking $u=1,v=g_2^{kr}$ in the definition of the equivalence relation we obtain the above relation is equivalent with
\[ (g_0,g_1,g_2).\]
\end{proof}

\begin{proposition}\label{prpCur} Suppose $(p_0,p_1,p_2)$ is a reduced type. Let $r=\gcd(p_0,p_1,p_2)$ and let $s_{01}=\gcd(p_0/r,p_1/r)$. Then there is a map from the set of  quasi-toric relations of type $(p_0,p_1,p_2)$ to the set of $\C(x_1,x_2)$-rational points of the curve
\[ \frac{F_0(1,x_1,x_2)}{F_2(1,x_1,x_2)} z_1^{p_0/s_{01}}+\frac{F_1(1,x_1,x_2)}{F_2(1,x_1,x_2)} z_2^{p_1/s_{01}}+1=0\]
such that each non empty fiber consists of precisely $s$ points.
 
If $p_2=\lcm(p_0,p_1)$, i.e., if $s_{01}=1$ then this map is also surjective and therefore bijective.
\end{proposition}
\begin{proof} %Since the type is reduced we can write\[ p_0=r s_{01} s_{02} ,p_1=rs_{01}{s_{12}}, p_2=rs_{02}{s_{12}}.\]
%Moreover we may assume  that $v\geq t\geq s$, in order to have $p_0\leq p_1\leq p_2$.
%
Take a  quasi-toric relation $(h_0,h_1,h_2)$. Then we can rewrite \ref{eqnTorRel}
\[ \left(  \frac{h_0^{s_{01}}}{h_2^{s_{12}}}  \right)^{r{s_{02}}} F_0+ \left(  \frac{h_1^{s_{01}}}{h_2^t } \right)^{r{s_{12}}} F_1+F_2=0.\]
I..e., this relation yields a $\C(x_1,x_2)$-point
\[ \left( \frac{h_0^{s_{01}}}{h_2^{s_{12}}},\frac{h_1^{s_{01}}}{h_2^{s_{02}}}\right)\]
 on the affine curve with equation
\[ \frac{F_0(1,x_1,x_2)}{F_2(1,x_1,x_2)} z_1^{r{s_{02}}}+\frac{F_1(1,x_1,x_2)}{F_2(1,x_1,x_2)} z_2^{r{s_{12}}}+1=0\]
One easily checks that an equivalent relation is mapped to the same point.

Suppose now that $(g_0,g_1,g_2)$ yields the same point on $\cC$ as $(h_0,h_1,h_2),$ i.e.,  
\[\frac{g_0^{s_{01}}}{g_2^{s_{12}}}=\frac{h_0^{s_{01}}}{h_2^{s_{12}}}\mbox{ and }\frac{g_1^{s_{01}}}{g_2^{s_{02}}}=\frac{h_1^{s_{01}}}{h_2^{s_{02}}}\]
hold. The above equations can be rewritten as
\[ \left(\frac{g_0}{h_0}\right)^{s_{01}}=\left(\frac{g_2}{h_2}\right)^{s_{12}} \mbox{ and }\left(\frac{g_1}{h_1}\right)^{s_{01}}=\left(\frac{g_2}{h_2}\right)^{s_{02}}\]
Recall that $\gcd(s_{01},{s_{02}})=\gcd(s_{01},{s_{12}})=\gcd({s_{02}},{s_{12}})=1$.  In particular the first equation shows that $g_2/h_2$ is a $s_{01}$-th power.
I.e., there exists forms $f_1,f_2$ without a common factor such that $g_2/h_2=f_1^{s_{01}}/f_2^{s_{01}}$. From this it follows that
\[ \left(\frac{g_0}{h_0}\right)^{s_{01}}=\left(\frac{f_1}{f_2}\right)^{s_{01}{s_{12}}} \]
Let $\zeta=\exp(2\pi I/s_{01})$. Then 
\[ \frac{g_0}{h_0}=\frac{f_1^{s_{12}}}{f_2^{s_{12}}}\zeta^{k_0}\]
for some $k_0$. Similarly we obtain
\[ \frac{g_1}{h_1}=\frac{f_1^{s_{02}}}{f_2^{s_{02}}}\zeta^{k_1}\]
for some $k_1$.

However $f_1/f_2$ is only determined up to a $s_{01}$-th root of unity. So we may multiply $f_1$ by an appropriate $s_{01}$-th root of unity in order to obtain $k_0=0$.
From this it follows that $(g_0,g_1,g_2)$ is equivalent to $(h_0,h_1\zeta^k,h_2)$ for some $k$. Hence the fibers of the map have at most $s$ elements. Since $s_{01}$ divides $p_1$ we obtain that  for each $k$
$(h_0,h_1\zeta^k,h_2)$ is a quasi-toric relation for $(F_0,F_1,F_2)$ of type $(p_0,p_1,p_2)$. Moreover, 
$(h_0,h_1\zeta^k,h_2)$ and  $(h_0,h_1\zeta^m,h_2)$ are equivalent if and only if $k\equiv m\bmod s_{01}$. Hence every nonempty fiber consists of precisely $s_{01}$ elements.

Suppose now that $s_{01}=1$ then obviously the map is injective. Suppose we have a point $(g_0/g_1,g_2/g_3)$ of $\cC$. Then there exists polynomials $f_1,f_2,u$ such that $f_1g_1=u^{s_{12}}$ and $f_2g_3=u^{s_{02}}$. In particular,
\[ \left(\frac{g_0}{g_1},\frac{g_2}{g_2}\right)=\left( \frac{f_1g_0}{u^{s_{12}}},\frac{f_2g_2}{u^{s_{02}}}\right).\] is the point associated with the quasi-toric relation $(f_1g_0,f_2g_2,u)$. Hence the map is surjective.
\end{proof}

\begin{remark} In  \cite[Introduction]{TorDec} we gave a map from the quasi-toric relations to the homology of a threefold  $X$, under the assumptions that  $F_0=F_1=1$ and $p_2=\lcm(p_0,p_1)$. The above Proposition allows us to extend this construction to arbitrary reduced types $(p_0,p_1,p_2)$ and forms $(F_0,F_1,F_2)$.

In order to have quasi-toric relation we need that for $\{i,j,k\}=\{0,1,2\}$ that  $\deg(F_i)-\deg(F_j)$ is divisible by $rs_{ij}=p_i/s_{ik}$. Suppose this is the case.

Let $X\subset \Ps(w_1,w_2,1,1,1)$ be the hypersurface defined by 
\[ F_0(z_0,z_1,z_2)x^{p_0/s_{01}}+F_1(z_0,z_1,z_2)y^{p_1/s_{01}}+F_2=0\]
where $w_0=\frac{s_{01}(\deg(F_2)-\deg(F_0))}{p_0}$ and $w_1=\frac{s(\deg(F_2)-\deg(F_1))}{p_1}$.
Similarly as in the case where $s_{01}=1, F_0=F_1=1$ (cf. \cite[Introduction]{TorDec} we have maps 
\[ \left\{\begin{array}{c}\mbox{Quasi-toric relations}\\\mbox{of type }(p_0,p_1,p_2)\end{array}\right\} \to \cC(\C(x_1,x_2))\to \Jac(\cC)(\C(x_1,x_2))\to H_4(X,\Z)_{\prim}\]
However, if $r>1$ then the first map is not bijective. This map may fail to be surjective and is definitely not injective, but has finite fibers. The second map is injective if $\cC$ has genus at least 1. The third map is  injective.
\end{remark}

\begin{remark}
In the case $r=1$ the above curve $\cC$ is isomorphic with
\[ z_1^{p_0/s_{01}}+ z_2^{p_1/s_{01}}+F_0^aF_1^bF_2=0\]
with $a\equiv -1 \bmod (p_0/s_{01})$, $a\equiv 0 \bmod p_1/s_{01}$, $b\equiv -1\bmod p_0/s_{01}$, $b\equiv 0 \bmod p_1/s_{01}$. (Such integers exists because $\gcd(p_0/s_{01},p_1/s_{01})=r=1$.)

In this case we can use Thom-Sebastiani to relate the dimension of $H_4(X,\Z)$ with the zeroes of the Alexander polynomial of $F_0^aF_1^bF_2$. In particular, the existence of quasi-toric structures force the Alexander polynomial to be non-constant. Similarly if $r>1$ and $F_0=F_1=1$ then we can use Thom-Sebastiani to relate the existence of quasi-toric structures to zeros of the Alexander polynomial of the curve $F_2=0$.

If $r>1$ and at least one of $F_0, F_1$ is non-constant then we cannot separate the variables and therefore not apply Thom-Sebastiani. We are not aware of a proof relation quasi-toric structures with the Alexander polynomial in this case.
\end{remark}

%$\cC(\C(x_1,x_2))\to H_4(X,\Z)_{\prim}$ as explained in \cite[Introduction]{TorDec}.

%In \cite{} we restricted ourselves to the case where $\gamma=1$. Under this additional assumption we have that each $\C(x_1,x_2)$-point yields an equivalence class of quasi-toric relations. 

%\end{proof}

\begin{proposition}\label{prpInf} 
Suppose $(p_0,p_1,p_2)$ is a reduced type, such that $p_0\leq p_1\leq p_2$. Then there exists a positive dimensional family of quasi-toric relations of $(p_0,p_1,p_2)$ parametrized by a rational variety if and only if one of the following occurs
\begin{enumerate}
\item $p_0=p_1=p_2=2$ and there is at least one quasi-toric relation or
\item $p_0=1$ and $p_1$ divides $\deg(F_1)-\deg(F_2)$.
\end{enumerate}
\end{proposition}

\begin{proof}
If $p_0=p_1=p_2=2$ and there is at least one $(2,2,2)$ quasi-toric relation then we can use a parametrization of the conic
\[ F_0X_0^2+F_1X_1^2+F_2X_2^2=0\]
to obtain infinitely many quasi-toric relations.

Suppose now that $p_0=1$. Since our types are reduced it follows that $p_1=p_2=p$. 
Let $g_2$ be a form of degree $N$. Let $g_1$ be a form of degree $N+\frac{\deg(F_1)-\deg(F_2)}{p}$ then
\[ (F_0^{p-1}(F_1g_1^{p}-F_2g_2^{p}),  F_0g_1,F_0g_2)\]
is a quasi-toric relation. 
This finishes the ``if" part of the proof.

For the other direction, suppose we have a family of $\C(x_1,x_2)$-rational points on the curve $\cC$ parametrized by a Zariski open of $\Ps^1_\C$. The equation of this curve yields an affine threefold $X$ and a morphism $\pi:X\to \As^2$. Each point in $\cC(\C(x_1,x_2)$ yields a rational section to $\pi$. If we have a family of rational sections thenon most fibers their image is distinct, and hence there is some smooth fiber of $\pi$ receiving  a map from some affine rational curve. In particular the compacification of a general fiber has genus 0.

%then we a family of points on the curve parametrised by $\Ps^1_\C$. 
This means that the curve $\cC$ over $\C(x_1,x_2)$ has genus 0. The projective closure of curve $\cC$ has a natural map to $\Ps^1$ of degree $p_0/s_{01}$, with at least $p_1/s_{01} $ ramification points of index $p_0/s_{01} $. To have genus 0 we need either $p_0/s_{01}=1$ or $p_1/s_{01}\leq 2$.

Recall that $p_0=rs_{01}s_{02}$. Hence in the first case we have that $r=s_{02}=1$. Using $p_0\leq p_1\leq p_2$ we obtain that $s_{01}\leq s_{02}$, hence also $s_{01}=1$  holds and therefore $p_0=1$.

If $p_1=s_{01}$ then using $p_0\leq p_1$ we obtain that $p_0/s_{01}=1$ and we are in the first case.

If $p_1=2s_{01}$ then we have two possibilities. First if $r=2$ and $s_{12}=1$ then by the above inequalities we find $s_{01}=s_{02}=1$. In particular $(p_0,p_1,p_2)=(2,2,2)$.
If $r=1$ and $s_{12}=2$ then $s_{01},s_{02}\leq s_{12}$ and $\gcd(s_{01}s_{02},s_{12})=1$. Hence $s_{01}=s_{02}=1$. In this case we have exponents $(1,2,2)$.

Recall that if $p_0=1$ then we have $r=s_{01}=s_{02}=1$. Hence $p_1=p_2=s_{12}=:p$.
In order to obtain that the three summands
\[ F_0 h_0, F_1h_1^p,F_2h_2^p\]
are of the same degree we need that $p$ divides $\deg(F_1)-\deg(F_2)$.
\end{proof}

\begin{remark}\label{remCom}
In \cite{TT} constructed families with $(2,2,2n+1)$ quasi-toric relation. The above result shows that there is a different way to obtain $(2,2,2n+1)$ relations. Starting with a $(2,2,1)$ 
relation  $(g_0,g_1,g_2)$ we obtain $(g_2^ng_0, g_2^ng_1,g_2)$ as a quasi-toric relation. 

The point is that the author manages to choose $F_0,F_1,F_2$ such that the $g_2$ is a perfect $2n+1$ power, in other words they manages to find a $(2,2,2n+1)$ relation such that $\gcd(g_0,g_2)=\gcd(g_1,g_2)=1$.
\end{remark}
\bibliographystyle{plain}
\bibliography{remke2}
\end{document}